\documentclass[11pt,reqno]{amsart}
\usepackage{amsmath}
\usepackage{amssymb}
\usepackage{amsthm}
\usepackage{amsrefs}
\usepackage{amsfonts}
\usepackage[dvipsnames]{xcolor}
\usepackage{mathtools}
\usepackage{mathrsfs}
\usepackage{extarrows}
\usepackage{xcolor}
\usepackage{framed}
\newif\ifrevmarks
\revmarkstrue
\ifrevmarks
  
  \newenvironment{revised}%
    {%
     \MakeFramed{\FrameRestore}}%
    {\endMakeFramed}
\else

\fi
\usepackage{geometry}
\usepackage{hyperref}
\usepackage{fancyhdr}
\usepackage{parskip}
\newtheorem{thm}{Theorem}[section]
\newtheorem{prop}[thm]{Proposition}
\newtheorem{lem}[thm]{Lemma}

\theoremstyle{remark}

\newcommand{\SO}{\operatorname{SO}}

\newcommand{\cX}{\mathcal{X}}

\newcommand{\cV}{\mathcal{V}}
\newcommand{\Q}{\mathbb{Q}}
\newcommand{\bR}{\mathbb{R}}
\newcommand{\bZ}{\mathbb{Z}}
\newcommand{\ra}{\rightarrow}
\newcommand{\eps}{\varepsilon}
\newcommand{\bT}{\textit{\textbf{T}}}
\DeclareMathOperator{\Wei}{Wei}
\DeclareMathOperator{\id}{Id}
\DeclareMathOperator{\Poi}{Poi}

\newcommand{\qand}{\quad \textrm{and} \quad}

\DeclareMathOperator{\supp}{supp}

\begin{document}

\title[Extremes for intrinsic Diophantine approximation on spheres]{Extreme value laws for intrinsic Diophantine approximation on spheres}

\author{Alexander Gorodnik}
\address{Institut f\"ur Mathematik, Universit\"at Z\"urich, 8057 Z\"urich, Switzerland}
\email{alexander.gorodnik@math.uzh.ch}

\author{Zouhair Ouaggag}
\address{Institut f\"ur Mathematik, Universit\"at Z\"urich, 8057 Z\"urich, Switzerland}
\email{zouhair.ouaggag@math.uzh.ch}


\begin{abstract}
    We study intrinsic Diophantine approximation on the sphere $S^d$, measured by
    the quality of the best rational approximation available at a given height. We
    show that, for $d\geq 3$ and along sufficiently sparse sequences of heights, these quantities have a
    Weibull limit distribution and the associated counting functions are
    asymptotically Poissonian.
\end{abstract}

\maketitle

\section{Introduction}

Let $d\geq 1$ and let $S^d\subset \bR^{d+1}$ denote the unit sphere. This paper
concerns \emph{intrinsic} Diophantine approximation on $S^d$, that is, the
approximation of points of $S^d$ by rational points on the sphere $S^d$.
Kleinbock and Merrill \cite{kleinbock2013rational}*{Theorem 4.1} established an
analogue of Dirichlet's Theorem in this setting: there is a constant $\kappa\geq 1$, depending only on $d$,
such that for every $\alpha\in S^d$ and every $N>1$ there exist $p\in\bZ^{d+1}$ and $q\in\mathbb{N}$ such that
\begin{equation}\label{dirichlet}
    \Big\lVert \alpha-\frac{p}{~q~} \Big\rVert<\frac{\kappa}{~\sqrt{qN}~},
    \qquad  \frac{\;p\;}{q}\in S^d,
    \qquad q\leq N.
\end{equation}
It is therefore natural to measure the quality of the best approximation
available at a given height by the minima
\begin{equation}\label{def kT}
    k_{t}(\alpha):=\min\Big\{ \sqrt{q\,e^t}\,\Big\lVert \alpha-\frac{p}{~q~} \Big\rVert\; : \:
    p\in \bZ^{d+1},\; q\in\mathbb{N},\; \frac{\;p\;}{q}\in S^d,\; q< e^t  \Big\},
    \qquad \alpha\in S^d .
\end{equation}
We use throughout the convention that a minimum over the empty set is $+\infty$.
By \eqref{dirichlet},
\begin{equation}\label{eq:dir2}
\kappa_d:=\sup\big\{k_t(\alpha)\;:\;\alpha\in S^d,\; t> 0\big\}<\infty .
\end{equation}

The function $t\mapsto k_t(\alpha)$ fluctuates as $t\to\infty$, and our aim is to
describe the distribution of its minima along sparse sets of times. For a finite
subset $\bT$ of $\bR_+$ we put
\begin{equation}\label{minimum kT tilde}
   \widetilde{k}_{\bT}(\alpha) := \min_{t \in \bT} \, k_t(\alpha).
\end{equation}
We show that, along sufficiently sparse sequences $(\bT_n)$ of sampling sets,
the suitably normalised minima $\widetilde k_{\bT_n}$ have a nondegenerate limit
distribution with respect to the rotation-invariant probability measure
$\mu_{_{S^d}}$ on $S^d$. Recall that the \emph{Weibull distribution}
$\Wei(\lambda,a)$ with scale $\lambda>0$ and shape $a>0$ is the distribution on
$[0,\infty)$ with distribution function $W_{\lambda,a}(\xi):=1-\exp\big(-(\xi/\lambda)^a\big)$.

\begin{thm}\label{th:weibull}
	Let $d\geq 3$. There exists $m_d>0$, depending only on $d$, such that for every sequence
	$(\bT_n)$ of finite subsets of $\bR_+$ with
	\begin{equation}\label{eq:sparse}
	|\bT_n|\to \infty \qand
 \frac{\displaystyle \min_{t\ne t'\in \bT_n} \left\{t,t',|t-t'| \right\}}{\log |\bT_n|} \to \infty,
	\end{equation}
	we have
	\[
	|\bT_n|^{1/d} \cdot \widetilde k_{\bT_n} \xLongrightarrow[\mu_{_{S^d}}]{} \Wei\big(m_d^{-1/d},d\big).
	\]
\end{thm}

Theorem \ref{th:weibull} is a consequence of a finer result, which describes the
number of times at which $k_t$ is small. For a finite subset $\bT$
of $\bR_+$ and $\tau>0$ we consider the counting function
\begin{equation}\label{def K}
\mathcal{K}_{\bT}(\alpha,\tau):= \#\Big\{t\in \bT \; : \; k_t(\alpha)<\tau \Big\}.
\end{equation}
We show that along sparse sequences $(\bT_n)$ this counting function is asymptotically
Poissonian.

\begin{thm}\label{th:poisson}
Let $d\geq 3$ and let $m_d>0$ be as in Theorem \ref{th:weibull}. For every sequence $(\bT_n)$ of
finite subsets of $\bR_+$ satisfying \eqref{eq:sparse} and every $\tau> 0$,
$$
   \mathcal{K}_{\bT_n}\big(\cdot\,,|\bT_n|^{-1/d} \tau\big)\xLongrightarrow[\mu_{_{S^d}}]{} \Poi(m_d\tau^d),
$$
where $\Poi(m)$ denotes the Poisson distribution with mean $m$.
\end{thm}

Theorem \ref{th:weibull} also follows directly from Theorem \ref{th:poisson}: since
$$
\Big\{\widetilde k_{\bT_n}<\tau\Big\} =\bigcup_{t\in \bT_n}\Big\{k_t<\tau\Big\}
=\Big\{\mathcal{K}_{\bT_n}(\cdot,\tau)>0\Big\},
$$
taking $\tau=|\bT_n|^{-1/d}\xi$ in this identity and applying
Theorem \ref{th:poisson} gives, for every $\xi>0$,
$$
\mu_{_{S^d}}\Big(\Big\{|\bT_n|^{1/d}\widetilde k_{\bT_n}<\xi\Big\}\Big)
=\mu_{_{S^d}}\Big(\Big\{\mathcal{K}_{\bT_n}\big(\cdot,|\bT_n|^{-1/d}\xi\big)>0\Big\}\Big)
\longrightarrow 1-e^{-m_d\xi^d}.
$$

The connection between homogeneous dynamics and Diophantine approximation goes back at least to the work of Sullivan
\cite{Sullivan1982}, where the interplay between cuspidal excursions
of geodesics on hyperbolic manifolds and Diophantine approximation was explored;
in particular, Khintchine's theorem was related to the logarithm law for geodesics.
This profound connection was further developed by 
Dani \cite{Dani1985} and Kleinbock and Margulis \cite{KleinbockMargulis1998},
\cite{Kleinbock_1999}. 
The distributional refinement of such geodesic excursion results,
which is our concern here, goes back to Pollicott \cite{Pollicott2009}, who
proved an extreme-value law for cusp excursions of the geodesic flow on the
modular surface.
In the setting of the sphere, the relevant correspondence---an analogue of
Dani's correspondence \cite{Dani1985} relating good rational approximants of $\alpha\in S^d$ to
short vectors of an associated lattice on the light cone in $\bR^{d+2}$---is
implicit in \cite{drutu2005} and explicit in \cite{kleinbock2013rational}.
We also refer to subsequent works
\cite{alam2020quantitative}, \cite{KelmerYu2023}, \cite{ouaggag2023effective},
\cite{ouaggag2024clt}.
Poisson limit laws in dynamics go back to Doeblin \cite{Doeblin1940}, who
determined the limiting distribution of the return times of the Gauss map to
shrinking neighbourhoods of the origin; the subject was taken up again by Pitskel \cite{Pitskel1991} and Hirata
\cite{Hirata_1993}.
Poisson laws on the space of lattices, and the limit laws for the associated
arithmetic functions, were established by Dolgopyat, Fayad and Liu
\cite{dolgopyat2021multipleborelcantellilemma} and by Bj\"orklund and Gorodnik
\cite{WeibullBG}.

We follow the approach of \cite{WeibullBG}, but its implementation in the
present setting requires new ingredients. The first is the dynamical
correspondence developed in Section \ref{sec:correspondence}, which expresses the
Diophantine extremes $k_t(\alpha)$ in terms of the space of orthogonal lattices
and involves an auxiliary parameter $\sigma_0$, whose appearance is somewhat
surprising: it does not enter the statements of our results and its powers cancel
at the end of the proof. The second is a different
treatment of the neighbourhoods of the cusps, carried out in Section
\ref{sec:volume} by means of precise reduction theory. Finally, the proof
rests on the effective multiple equidistribution in the space of orthogonal
lattices established by the second author in \cite{ouaggag2024clt}. This last
ingredient is the source of the assumption $d\geq 3$ in Theorems \ref{th:weibull}
and \ref{th:poisson}, under which the equidistribution estimate of
\cite{ouaggag2024clt} is proved; the other arguments of the paper are valid for
all $d\geq 1$.

\subsection*{Organisation of the paper}
Section \ref{sec:notation} fixes the notation. In Section \ref{sec:correspondence}
we set up the correspondence between good
rational approximants on $S^d$ and short vectors of lattices on the light cone,
and show (Lemmas \ref{lem:tilde kt tilde etaT} and
\ref{lem:poisson tilde kt tilde etaT}) that the minima $\widetilde k_{\bT_n}$ and
the corresponding lattice minima have related limit distributions along sparse
sequences, and similarly for the associated counting functions. Section \ref{sec:volume} contains the volume estimates for the
shrinking targets on the cone and computes, by means of precise reduction theory,
the volume of the associated neighbourhoods of infinity in the space of lattices.
In Section \ref{sec:Poisson}
we combine these estimates with effective multiple equidistribution to prove
Theorems \ref{th:weibull} and \ref{th:poisson}.

\subsection*{Acknowledgements}
This research was supported by the SNF grant 200020--212617.
The draft of the paper was proofread and polished by Claude.

\section{Notation}\label{sec:notation}

We consider the quadratic form
\begin{equation*}
Q(x) := \sum_{i=1}^{d+1} x_i^{2} - x_{d+2}^{2}, \quad \text{ for } x= (x_1, \dots , x_{d+2}),
\end{equation*}
and the corresponding positive light cone
$$
\mathcal{V} := \{  x \in \mathbb{R}^{d+2} : Q(x)=0,~ x_{d+2} > 0 \}.
$$
Let
$$
G=\SO(Q)^\circ \cong \SO(d+1,1)^\circ
$$
be the identity component of the group of linear transformations of $\mathbb{R}^{d+2}$ preserving $Q$.
Note that $G$ acts transitively on $\mathcal{V}$.
We write
$$
\mathcal{V}(\mathbb{Z}):=\mathcal{V}\cap\mathbb{Z}^{d+2}
=\Big\{(p,q)\in \bZ^{d+1}\times\mathbb{N} \;:\; \frac{\;p\;}{q}\in S^d\Big\},
$$
the set of rational points occurring in \eqref{def kT}.
We call the sets $g\mathcal{V}(\mathbb{Z})$, $g\in G$, \emph{lattices} in $\mathcal{V}$.
Denote by $\Gamma$ the stabiliser of $\mathcal{V}(\mathbb{Z})$ in $G$. Then $\Gamma$ contains the group $\SO(Q)^\circ(\bZ)$ of integral points of $G$ as a subgroup of finite index, and is in particular a lattice in $G$.
Therefore, the $G$-orbit of $\mathcal{V}(\mathbb{Z})$ is identified with
$$
\mathcal{X} :=  G / \Gamma,
$$
which we endow with the unique $G$-invariant probability measure $\mu$. We
denote by $\mu_G$ the Haar measure of $G$, normalised compatibly with $\mu$.

Let $K$ denote the stabiliser of the last basis vector $e_{d+2}$ in $G$, a maximal
compact subgroup of $G$; it consists of the matrices
$$
\begin{pmatrix}
k &  \\
 & 1
\end{pmatrix},\qquad k\in\SO(d+1),
$$
so that $K\cong\SO(d+1)$, and we equip it with the invariant probability measure $\mu_K$.
We set $e_0:=(0,\dots,0,1,1)\in\cV$ and $M:=\operatorname{Stab}_K(e_0)\cong \SO(d)$; the
sphere $S^d$ can then be realised as the quotient $K/M$.

We embed the sphere $S^{d}$ in $\mathcal{V}$
via the map $S^d\rightarrow \mathcal{V},\;\alpha \mapsto (\alpha,1)$.
Under this identification, good rational approximations of a point
$\alpha\in S^d$ correspond to integer points $(p,q) \in \mathcal{V}(\mathbb{Z})$
lying close to the ray through $(\alpha,1)$ on the light cone $\mathcal{V}$.
For
$\alpha\ne\pm e_{d+1}$ we denote by $k_\alpha\in K$ the rotation which maps $\alpha$ to
$e_{d+1}$ in the plane spanned by these two vectors and is the identity on the
orthogonal complement of this plane, so that
\begin{equation}\label{def k alpha}
    k_{\alpha}(\alpha,1)=e_0
\end{equation}
(for $\alpha=\pm e_{d+1}$ we make an arbitrary choice of $k_\alpha\in K$
satisfying \eqref{def k alpha}). We introduce the lattice
\begin{equation}
\label{def Lambda alpha}
\Lambda_\alpha:=  k_\alpha \mathcal{V}(\mathbb{Z})\in \cX.
\end{equation}
We denote by $\nu$ the probability measure on $\mathcal{X}$, defined by
\[
\int_{\mathcal{X}} \varphi \,   d\nu = \int_{S^d} \varphi(\Lambda_\alpha) \,  d\mu_{_{S^d}}(\alpha)
,  \quad \textrm{for $\varphi \in C_b(\mathcal{X})$}.
\]
We also consider the subgroup
$$
A:=\left\{a_t =  \begin{pmatrix}
I_d & & \\
 &  \cosh t & -\sinh t \\
 & -\sinh t & \cosh t \end{pmatrix} : t\in \bR\right\} \subset G,
$$
which contracts the line spanned by $e_0$ as $a_te_0=e^{-t}e_0$, and the maximal
unipotent subgroup of $G$ fixing $e_0$,
$$
U:=\left\{u_{x}=\left(\begin{matrix} I_d & -x & x\\
 x^t&1-\frac{1}{2}\lVert x \rVert^2  &\frac{1}{2}\lVert x \rVert^2 \\
 x^t& -\frac{1}{2}\lVert x \rVert^2& 1+\frac{1}{2}\lVert x \rVert^2\end{matrix}\right): x\in \mathbb{R}^d\right\}.
$$
Then $G=KAU$ is an Iwasawa decomposition, and $P:=MAU$ is the parabolic
subgroup of $G$ stabilising the line $\bR e_0$. Every $g\in G$ can be written
uniquely as $g=ka_tu_{x}$ with $k\in K$, $t\in\bR$ and $x\in\bR^{d}$, and in these
coordinates
\begin{equation}\label{decomposition Haar muG}
d\mu_G(g)=c_G\,e^{-dt}\, dx\, dt\, d\mu_K(k)
\end{equation}
for some constant $c_G>0$, where $dx$ and $dt$ denote the Lebesgue measures on $\mathbb{R}^d$ and $\mathbb{R}$.

The stabiliser of $e_0$ in $G$ is $MU$, and the cone $\mathcal{V}\cong G/MU$
carries a $G$-invariant measure $\mu_{\mathcal{V}}$, unique up to scaling. Every
$x\in\mathcal{V}$ can be written uniquely as $x=ka_te_0=e^{-t}ke_0$ with $kM\in K/M$
and $t\in\bR$, and we normalise $\mu_{\mathcal{V}}$ so that in these polar coordinates
\begin{equation}\label{decomposition muV}
d\mu_{\mathcal{V}}(ka_te_0)=e^{-dt}\,dt\,d\mu_{K/M}(kM),
\end{equation}
where $\mu_{K/M}$ denotes the $K$-invariant probability measure on $K/M$.

Finally, we introduce the linear forms
$$
L_+(x):=x_{d+2}- x_{d+1}\qand L_-(x):=x_{d+2}+ x_{d+1}.
$$
Note that
\begin{equation}
L_\pm\ge 0\;\;\hbox{ on $\cV$},\quad\quad
    L_+(a_t x)=e^{t}L_+(x), \quad\quad L_-(a_t x)=e^{-t}L_-(x),\label{eq:scalling L_pm by a_t}
\end{equation}
and that $L_+(x)=0$ if and only if $x$ is a positive multiple of $e_0$.
For $x\in \cV$, $\sum_{i=1}^{d}x_i^2=L_+(x)L_-(x)$,
$x_{d+2}=\tfrac{1}{2}\big(L_+(x)+L_-(x)\big)$ and
$x_{d+1}=\tfrac{1}{2}\big(L_-(x)-L_+(x)\big)$, so that
\begin{equation}\label{eq:norm}
\|x\|_\infty \le \max(L_+(x),L_-(x)),\qquad x\in\cV .
\end{equation}

\section{Good rational approximants and minimal lattice vectors}\label{sec:correspondence}

We shall show that best rational approximations for $\alpha\in S^d$ can be computed in
terms of the orbit $(a_t\Lambda_\alpha)_{t\ge 0}$. The key to this correspondence is
the following function on the space $\cX$, defined for a parameter
$\sigma_0\in(0,1]$, fixed once and for all and subject
to the condition imposed at the end of Section \ref{sec:volume}:
\begin{equation}\label{eq:eta}
 \eta(\Lambda):=\min\big\{ L_+(x):\,\,  L_-(x)< \sigma_0,\, x\in\Lambda\big\},\qquad \Lambda\in\cX .
\end{equation}
It follows from \eqref{eq:norm} that this minimum is attained.

\begin{lem}\label{lem:compare}
Set $\kappa_d':=\kappa_d^2/\sigma_0$. For all $\alpha\in S^d$ and $t>\ln(2/\sigma_0)$,
$$
\big(\sigma_0 +  e^{-2t} \kappa_d'\big)^{-1}
k^2_{t+\ln (\sigma_0 +  e^{-2t} \kappa_d')-\ln 2}(\alpha)\le
 \eta(a_t\Lambda_\alpha)\le \sigma_0^{-1}\,k^2_{t+\ln \sigma_0-\ln 2}(\alpha).
$$
\end{lem}

\begin{proof}
Recall that $\Lambda_\alpha=k_\alpha \mathcal{V}(\mathbb{Z})$. For $(p,q)\in\mathcal{V}(\mathbb{Z})$,
we write 
$$
k_\alpha(p,q)=(\xi_1,\dots,\xi_{d+1},q).
$$ 
Then since $k_\alpha$ is an isometry,
$\xi_1^2+\cdots+\xi_{d+1}^2=p_1^2+\cdots+p_{d+1}^2=q^2$,
and
\begin{align*}
   q^2\Big\lVert \alpha-\frac{p}{~q~} \Big\rVert^2
    &= \Big\lVert q\alpha-p\Big\rVert^2=\Big\lVert k_\alpha\left(q\alpha-p\right)\Big\rVert^2\\
    &= \Big\lVert (-\xi_1,\dots,-\xi_d, q-\xi_{d+1},0)\Big\rVert^2
     =  2q(q-\xi_{d+1}).
\end{align*}
Therefore,
$$
L_+\big(a_tk_\alpha(p,q)\big)=e^t(q-\xi_{d+1})=\frac{1}{2}\,q\,e^t\Big\lVert \alpha-\frac{p}{~q~} \Big\rVert^2.
$$
Also
$$
L_-\big(a_tk_\alpha(p,q)\big)=e^{-t}(q+\xi_{d+1})=e^{-t}\Big(2q-e^{-t} L_+\big(a_t k_\alpha(p,q)\big)\Big).
$$
We conclude that 
\begin{align} \label{eq:minima-eta}
     \eta(a_t\Lambda_\alpha)
     &=\min\Big\{ L_+\big(a_tk_\alpha(p,q)\big)\; : \:(p,q)\in\mathcal{V}(\mathbb{Z}),\; 
q< \sigma_0 e^{t}/2 + e^{-t} L_+\big(a_t k_\alpha(p,q)\big)/2
\Big\}\\
     &=\frac{1}{2}\min\Big\{ q\,e^t\Big\lVert \alpha-\frac{p}{~q~} \Big\rVert^2\; : \:(p,q)\in\mathcal{V}(\mathbb{Z}),\; 
q< \sigma_0 e^{t}/2 + e^{-t} L_+\big(a_t k_\alpha(p,q)\big)/2
\Big\}.\nonumber
\end{align}
Since $L_+\ge 0$, restricting the range of $q$ in \eqref{eq:minima-eta} can only increase the minimum, so that
$$
 \eta(a_t\Lambda_\alpha)\le
 \frac{1}{2}\min\Big\{ q\,e^t\Big\lVert \alpha-\frac{p}{~q~} \Big\rVert^2\; : \:(p,q)\in\mathcal{V}(\mathbb{Z}),\;
q< \sigma_0 e^{t}/2
\Big\}.
$$
With $s:=t+\ln\sigma_0-\ln 2$, so that $e^{s}=\sigma_0e^{t}/2$, the right-hand side
equals $\sigma_0^{-1}k_s^2(\alpha)$, which gives one of the inequalities.
In particular, $k_s(\alpha)\le \kappa_d$ by \eqref{eq:dir2} when $s> 0$, so that
$$
 \eta(a_t\Lambda_\alpha)\le \sigma_0^{-1}\kappa_d^2=\kappa_d'\quad\hbox{for all $\alpha\in S^d$ and $t>\ln(2/\sigma_0)$.}
$$

To prove the other inequality, we pick $(p_0,q_0)\in \mathcal{V}(\mathbb{Z})$ where the minimum in \eqref{eq:minima-eta} is achieved.
Then 
\begin{align*}
q_0 < \sigma_0 e^{t}/2 + e^{-t} L_+\big(a_t k_\alpha(p_0,q_0)\big)/2=
\sigma_0 e^{t}/2 +  e^{-t}  \eta(a_t\Lambda_\alpha)/2
\le
e^{t}(\sigma_0 +  e^{-2 t} \kappa_d')/2.
\end{align*}
Therefore,
$$
 \eta(a_t\Lambda_\alpha)=\frac{1}{2} q_0\,e^t\Big\lVert \alpha-\frac{p_0}{q_0} \Big\rVert^2
\ge \big(\sigma_0 +  e^{-2t}\kappa_d'\big)^{-1}
k_{t+\ln (\sigma_0 +  e^{-2t}\kappa_d')-\ln 2}^2(\alpha),
$$
which proves the second inequality.
\end{proof}

The relation in Lemma \ref{lem:compare} will allow us to compare the minima of the functions $k_t$ and $\eta_t$ over subsets of $\bR_+$ and, under a sparseness condition on the samples, to relate their limit distributions with respect to $\mu_{_{S^d}}$. Since $\min \bT_n\to\infty$ under \eqref{eq:sparse}, Lemma \ref{lem:compare} is applicable at all the times occurring in the proofs below once $n$ is large, and we use it there without further comment. For simplicity, we write
$$
\eta_t(\alpha):=\eta(a_t\Lambda_\alpha)
$$
and recall that $\nu$ is the push-forward of $\mu_{_{S^d}}$ under $\alpha\mapsto\Lambda_\alpha$, so that a limit law with respect to $\nu$ for a function on $\mathcal{X}$ is equivalent to the corresponding limit law with respect to $\mu_{_{S^d}}$ for its composition with $\alpha\mapsto\Lambda_\alpha$.
We define, for any finite subset $\bT$ of $\bR_+$, the minima
$$
    \widetilde{\eta}_{\bT}(\Lambda):=\min_{t\in \bT} \eta(a_t\Lambda),\quad \Lambda\in \cX,
    \qquad\hbox{and}\qquad
    \widetilde{\eta}_{\bT}(\alpha):=\widetilde{\eta}_{\bT}(\Lambda_\alpha),\quad \alpha\in S^d,
$$
in analogy with $\widetilde{k}_{\bT}$ in \eqref{minimum kT tilde}.
We show in the remainder of this section that the minima $\widetilde{k}_{\bT_n}$ and $\widetilde{\eta}_{\bT_n}$ have related limit $\mu_{_{S^d}}$-distributions along sparse sequences $(\bT_n)$, and that the same is true of the corresponding counting functions. We begin with the minima, where the argument is simpler.

\begin{lem}\label{lem:tilde kt tilde etaT}
Let $\lambda>0$. Suppose that for all sequences $(\bT_n)$ of finite subsets of $\bR_+$ satisfying the sparseness condition \eqref{eq:sparse},
$$|\bT_n|^{2/d}\cdot\widetilde{\eta}_{\bT_n} \xLongrightarrow[\mu_{_{S^d}}]{} \Wei(\lambda, d/2).
$$
Then for all such sequences $(\bT_n)$,
$$|\bT_n|^{1/d}\cdot\widetilde{k}_{\bT_n} \xLongrightarrow[\mu_{_{S^d}}]{} \Wei(\sigma_0^{1/2}\lambda^{1/2}, d).
$$
\end{lem}

\begin{proof}
We need to show that for every $\xi>0$, 
$$
\mu_{_{S^d}}\left(\left\{\alpha\in S^d:\, |\bT_n|^{1/d}\cdot\widetilde{k}_{\bT_n}(\alpha)<\xi\right\}\right)
\to W_{\sigma_0^{1/2}\lambda^{1/2},d}(\xi),
$$
where $W_{\lambda,a}$ denotes the distribution function of $\Wei(\lambda,a)$.
It follows from Lemma \ref{lem:compare} that 
for
\begin{equation}\label{eq:t1}
\bT_n':=\{t-\ln\sigma_0+\ln 2:\, t\in \bT_n\},
\end{equation}
we have
$$
\mu_{_{S^d}}\left(\left\{\alpha\in S^d:\, |\bT_n|^{1/d}\cdot\widetilde{k}_{\bT_n}(\alpha)<\xi\right\}\right)
\le \mu_{_{S^d}}
\left(\left\{\alpha\in S^d:\, |\bT'_n|^{2/d}\cdot\sigma_0 \widetilde{\eta}_{\bT'_n}(\alpha)<\xi^2\right\}\right).
$$
Since the sets $\bT_n'$ also satisfy \eqref{eq:sparse}, it follows from our assumption that
$$
\mu_{_{S^d}}\left(\left\{\alpha\in S^d:\, |\bT'_n|^{2/d}\cdot\sigma_0 \widetilde{\eta}_{\bT'_n}(\alpha)<\xi^2\right\}\right) \to W_{\lambda,d/2}(\sigma_0^{-1}\xi^2)=W_{\sigma_0^{1/2}\lambda^{1/2},d}(\xi).
$$
To obtain the lower bound, we consider the sets 
\begin{equation}\label{eq:t2}
\bT_n'':=\{\psi^{-1}(t):\, t\in \bT_n\},
\end{equation}
where $\psi(t):= t+\ln (\sigma_0 +  e^{-2t} \kappa_d')-\ln 2$.
We note that this function is strictly increasing for sufficiently large $t$, so that \eqref{eq:t2} is well defined for $n$ large, $\min \bT_n$ tending to infinity, and 
$$
|\psi(t_1)-\psi(t_2)|=|t_1-t_2|+O\left(e^{-2\min(t_1,t_2)}\right),
$$
so that the sets $\bT_n''$ also satisfy \eqref{eq:sparse}.
Applying Lemma \ref{lem:compare}, we obtain that
\begin{align*}
&\mu_{_{S^d}}\left(\left\{\alpha\in S^d:\, |\bT_n|^{1/d}\cdot\widetilde{k}_{\bT_n}(\alpha)<\xi\right\}\right)\\
&\qquad\ge
\mu_{_{S^d}}\left(\left\{\alpha\in S^d:\, |\bT''_n|^{2/d}\cdot
\big(\sigma_0 +  e^{-2\min(\bT''_n)} \kappa_d'\big)
\widetilde{\eta}_{\bT''_n}(\alpha)<\xi^2\right\}\right).
\end{align*}
Let $\eps>0$. For sufficiently large $n$, we have $e^{-2\min(\bT''_n)}<\eps$,
so that it follows from our assumption that 
\begin{align*}
\mu_{_{S^d}}\left(\left\{\alpha\in S^d:\, |\bT_n|^{1/d}\cdot\widetilde{k}_{\bT_n}(\alpha)<\xi\right\}\right)&\ge
\mu_{_{S^d}}\left(\left\{\alpha\in S^d:\, |\bT''_n|^{2/d}\cdot
\widetilde{\eta}_{\bT''_n}(\alpha)<
(\sigma_0 +  \eps\kappa_d')^{-1} \xi^2\right\}\right) \\
&\to W_{\lambda,d/2}\left((\sigma_0 +  \eps\kappa_d')^{-1}\xi^2\right)= W_{\lambda^{1/2}(\sigma_0 +  \eps\kappa_d')^{1/2},d}\left(\xi\right).
\end{align*}
Since this holds for every $\eps>0$, using continuity we deduce  the correct lower bound.
\end{proof}

The above argument can be easily adapted to analyse the distribution of the number of approximations.
For a finite subset $\bT$ of $\bR_+$ and $\tau>0$, recall $\mathcal{K}_{\bT}(\alpha,\tau)$ from \eqref{def K} and set
\begin{equation}\label{def N}
\mathcal{N}_{\bT}(\Lambda,\tau):= \#\Big\{t\in \bT \; : \; \eta(a_t\Lambda)<\tau \Big\},\quad \Lambda\in \cX,
\end{equation}
writing again $\mathcal{N}_{\bT}(\alpha,\tau):=\mathcal{N}_{\bT}(\Lambda_\alpha,\tau)$ for $\alpha\in S^d$.
Similarly to Lemma \ref{lem:tilde kt tilde etaT}, the following lemma makes precise the relation between Poisson approximations of $\mathcal{N}_{\bT_n}$ and $\mathcal{K}_{\bT_n}$.

\begin{lem}\label{lem:poisson tilde kt tilde etaT}
Let $m>0$. Suppose that for all sequences $(\bT_n)$ of finite subsets of $\bR_+$ satisfying the sparseness condition \eqref{eq:sparse} and $\tau >0$,
$$
\mathcal{N}_{\bT_n}(\cdot,|\bT_n|^{-2/d} \tau)\xLongrightarrow[\mu_{_{S^d}}]{} \Poi(m\tau^{d/2}).
$$
Then for all such sequences $(\bT_n)$ and $\tau >0$,
$$ \mathcal{K}_{\bT_n}(\cdot,|\bT_n|^{-1/d} \tau)\xLongrightarrow[\mu_{_{S^d}}]{} \Poi(m\sigma_0^{-d/2}\tau^{d}).$$
\end{lem}

\begin{proof} 
According to our assumption, for all $\tau>0$ and all integers $k\geq 0$, 
\[
\mu_{_{S^d}}\Big(\left\{\alpha\in S^d:\,   \mathcal{N}_{\bT_n}(\alpha, |\bT_n|^{-2/d} \tau) \le  k \right\}\Big)
\longrightarrow \sum_{i=0}^k \frac{(m \tau^{d/2})^i}{i!} e^{-m \tau^{d/2}}.
\]
Because of Lemma \ref{lem:compare}, with
$\bT_n'$ as in \eqref{eq:t1},
we have $\sigma_0\eta_{t'}\le k_t^2$ for $t\in \bT_n$ and $t'=t-\ln\sigma_0+\ln 2\in \bT_n'$, so that
\begin{align*}
&\mu_{_{S^d}}\left(\left\{\alpha\in S^d:\, \mathcal{K}_{\bT_n}(\alpha, |\bT_n|^{-1/d} \tau) \leq k\right\}\right)\\
&\qquad= \mu_{_{S^d}}\left(\left\{\alpha\in S^d:\, \#\Big\{ t\in \bT_n\; : k_t(\alpha)^2< |\bT_n|^{-2/d} \tau^2\Big\}  \leq k\right\}\right) \\
&\qquad\mathrel{\geq} \mu_{_{S^d}}\left( \left\{\alpha\in S^d:\, \#\Big\{ t\in \bT_n'\; : \sigma_0\eta_t(\alpha)< |\bT_n'|^{-2/d} \tau^2\Big\}  \leq k\right\}\right) \\
&\qquad=\mu_{_{S^d}}\left(\left\{\alpha\in S^d:\, \mathcal{N}_{\bT'_n}(\alpha, \sigma_0^{-1}|\bT'_n|^{-2/d} \tau^2) \leq k\right\}\right).
\end{align*}
Similarly, with $\bT_n''$ as in \eqref{eq:t2} and when $e^{-2\min(\bT''_n)}
<\eps$, we have $k_t^2\le (\sigma_0 +  \eps\kappa_d')\eta_{t''}$ for $t\in \bT_n$ and $t''=\psi^{-1}(t)\in \bT_n''$, so that
\begin{align*}
&\mu_{_{S^d}}\left(\left\{\alpha\in S^d:\, \mathcal{K}_{\bT_n}(\alpha, |\bT_n|^{-1/d} \tau) \leq k\right\}\right)\\
&\qquad= \mu_{_{S^d}}\left(\left\{\alpha\in S^d:\, \#\Big\{ t\in \bT_n\; : k_t(\alpha)^2< |\bT_n|^{-2/d} \tau^2\Big\}  \leq k\right\}\right) \\
&\qquad\mathrel{\le} \mu_{_{S^d}}\left( \left\{\alpha\in S^d:\, \#\Big\{ t\in \bT_n''\; : (\sigma_0 +  \eps\kappa_d')\eta_t(\alpha)< |\bT_n''|^{-2/d} \tau^2\Big\}  \leq k\right\}\right) \\
&\qquad=\mu_{_{S^d}}\left(\left\{\alpha\in S^d:\, \mathcal{N}_{\bT''_n}\left(\alpha, (\sigma_0 +  \eps\kappa_d')^{-1}|\bT''_n|^{-2/d} \tau^2\right) \leq k\right\}\right).
\end{align*}
Our claim follows from these two estimates and our assumption using continuity.
\end{proof}

By Lemmas \ref{lem:tilde kt tilde etaT} and \ref{lem:poisson tilde kt tilde etaT},
both of our theorems will follow once we establish the Poisson limit for the counting
functions $\mathcal{N}_{\bT_n}$, which also yields the Weibull limit for
$\widetilde{\eta}_{\bT_n}$. This will be deduced in Section \ref{sec:Poisson} from
effective multiple equidistribution, combined with the volume estimates of Section~\ref{sec:volume}.

\section{Volume estimates}\label{sec:volume}

We recall from the previous section that the squares $k_t(\alpha)^2$
are comparable to the values $\eta(a_t\Lambda_\alpha)$, where $\eta$ is defined in \eqref{eq:eta}. It is therefore convenient to introduce, for $\tau,\sigma>0$, the sets
\begin{equation}\label{eq:C}
\mathcal{C}(\tau,\sigma):=\{x\in \mathcal{V}:\,\, L_+(x)<\tau,\; L_-(x)< \sigma\}.
\end{equation}
These sets increase in each of the two parameters, and, in particular, for
$\Lambda\in\cX$,
$$
\eta(\Lambda)<\tau \quad \Longleftrightarrow\quad \Lambda\cap \mathcal{C}(\tau,\sigma_0)\ne \emptyset.
$$
In this section we study properties of the sets $\mathcal{C}(\tau,\sigma)$
as well as the corresponding hitting sets 
\begin{equation}\label{def Omega}
\Omega(\tau,\sigma):= \{ \Lambda\in \mathcal{X}:\; \Lambda\cap \mathcal{C}(\tau,\sigma)\neq \emptyset\}.
\end{equation}

\begin{lem}\label{lem:volume regularity}
There exists $C_d>0$, depending only on $d$, such that
$$
\mu_\mathcal{V}\big(\mathcal{C}(\tau,\sigma)\big) =C_d (\tau\sigma)^{d/2}
\qquad\hbox{for all $\tau,\sigma>0$.}
$$
\end{lem}

\begin{proof}
We parametrise $\mathcal{V}$ by $x_1,\ldots,x_{d+1}$, the remaining coordinate being
determined by $Q(x)=0$ and $x_{d+2}>0$. Up to a positive constant, the invariant
volume is then given by $x_{d+2}^{-1}dx_1\ldots dx_{d+1}$.
We use the new coordinates
$$
y_1=x_1,\;\ldots,\; y_d=x_d,\; y_{d+1}=x_{d+2}+x_{d+1}=L_-(x).
$$
We have $(y_{d+1}-x_{d+1})^2 =\sum_{i=1}^d y_i^2 + x_{d+1}^2$, so that 
$$
x_{d+1} = -\frac{1}{2}y_{d+1}^{-1}\left(\sum_{i=1}^d y_i^2-y^2_{d+1}\right).
$$
Also,
$$
x_{d+2} = y_{d+1}-x_{d+1} = \frac{1}{2}y_{d+1}^{-1}\left(\sum_{i=1}^d y_i^2+y^2_{d+1}\right).
$$
Hence, since $x_{d+2}^{-1}=2y_{d+1}\big(\sum_{i=1}^d y_i^2+y_{d+1}^2\big)^{-1}$, the invariant volume in the new coordinates is given by
$$
2y_{d+1}\left(\sum_{i=1}^d y_i^2+y^2_{d+1}\right)^{-1} \cdot \left|\frac{\partial x_{d+1}}{\partial y_{d+1}}\right|  dy_1\ldots dy_{d+1}=y_{d+1}^{-1} dy_1\ldots dy_{d+1}.
$$
Since $L_+(x)=y_{d+1}^{-1}\sum_{i=1}^{d}y_i^2$, the domain of integration is given by the inequalities 
$$
0<y_{d+1}<\sigma\qand y_{d+1}^{-1}\sum_{i=1}^d y_i^2<\tau .
$$
Finally, we obtain 
$$
\int_0^{\sigma} y_{d+1}^{-1} \Big( \int_{\sum_{i=1}^d y_i^2 < \tau y_{d+1}} dy_1 \ldots dy_d \Big) dy_{d+1}
=\int_0^{\sigma} \omega_d \tau^{d/2} y_{d+1}^{d/2-1} dy_{d+1}=2d^{-1}\omega_d\, (\tau\sigma)^{d/2},
$$
where $\omega_d$ denotes the volume of the $d$-dimensional unit ball. This gives
the lemma.
\end{proof}

Given $\eps> 0$,  let
\begin{equation}
\label{def Weps}
W_\eps:= \big\{ g \in G : \| \, g - \id \|_{\textrm{op}} < \eps\big\},
\end{equation}
where $\|\cdot\|_{\textrm{op}}$ denotes the operator norm on $G$ with respect to the Euclidean norm. 

\begin{lem}
\label{lem:Weps}
There exists $c=c(d)>0$ such that
\[
W_\eps\cdot \mathcal{C}(\tau,\sigma) \subseteq \mathcal{C}(\tau+c\eps,\;\sigma+c\eps)\quad\hbox{for all $\eps>0$ and $\tau,\sigma \in (0,1]$.}
\]
\end{lem}

\begin{proof}
We have 
$$
W_\eps\cdot \mathcal{C}(\tau,\sigma)
\subseteq 
\big\{ x \in \mathcal{V} \,  : \,  \exists \,  y\in\mathcal{V} \enskip \textrm{such that} \enskip \|x-y\| < \eps\|y\|,   \enskip L_+(y) < \tau,  \enskip L_-(y) < \sigma \big\}.
$$
By \eqref{eq:norm}, every such $y$ satisfies $\|y\|_\infty< \max(\tau,\sigma)\le 1$, whence $\|y\|\le \sqrt{d+2}$.
Therefore,
$$
L_+(x) < \tau + |L_+(x) - L_+(y)|\le
\tau + \sqrt{2(d+2)}\,\eps,
$$
since $|L_\pm(v)|\le\sqrt{2}\,\|v\|$, and similarly
$L_-(x) <
\sigma + \sqrt{2(d+2)}\,\eps$.
The claim follows.
\end{proof}

The proper rational parabolic subgroups of $G$ are precisely the stabilisers of the
isotropic lines of $\Q^{d+2}$, and, by Witt's extension theorem, $G(\Q)$ acts
transitively on the set of these lines. Hence every proper rational parabolic subgroup of
$G$ is of the form $g^{-1}Pg$ with $g\in G(\Q)$.
The cusp regions of $\cX$ are parametrised by the $\Gamma$-conjugacy classes of proper
rational parabolic subgroups of $G$. Their number $s$ is finite, and we fix representatives
\begin{equation}\label{eq:cusps}
P_i=g_i^{-1}Pg_i,\qquad g_i\in G(\Q),\quad i=1,\dots,s .
\end{equation}
Following Garland and Raghunathan \cite{GarlandRaghunathan}, 
we parametrise the cusp regions by translates of
Siegel sets $\mathfrak{S}_{T,B}$. For $T\in\bR$ and a relatively
compact open subset $B$ of $U$ we set
$$
A_T:=\{a_t:\; t>T\}\qand \mathfrak{S}_{T,B}:=KA_TB.
$$
We write ${}^0P:=MU$, put
$$
\Gamma_i:=g_i\Gamma g_i^{-1}\cap\,{}^0P\qand \Lambda_i:=\Gamma_i\cap U ,
$$
and denote by $\pi:G\rightarrow\mathcal{X}=G/\Gamma$ the projection.
We note that $\Lambda_i$ is a cocompact lattice in $U$.
Let $M_i$ be the image of $\Gamma_i$ under the projection
${}^0P\rightarrow M$. It is finite because $M$ is compact, and the kernel of this map is $\Lambda_i$.

Since the line $g_i^{-1}\bR e_0$ is defined over $\Q$, we have
$$
\mathcal{V}(\bZ)\cap g_i^{-1}\bR e_0=\{m\,w_i:\; m\in \mathbb{N}\}
$$
for a unique $w_i\in \mathcal{V}(\bZ)$, and we define $c_i>0$ by
\begin{equation}\label{eq:ci}
g_iw_i=c_ie_0 .
\end{equation}

\begin{prop}\label{prop:reduction}
Let $B_i\subset U$ be bounded Borel fundamental domains
for $\Lambda_i$ and $W_i\subset K$ Borel fundamental domains  for
the right translation action of $M_i$ on $K$.	
There exists $T_0\geq 0$  such that the sets
$S_i:=W_iA_{T_0}B_i$ satisfy:
\begin{enumerate}
\item[(i)] $\pi$ is injective on $S_ig_i$ for every $i=1,\dots,s$;
\item[(ii)] the sets $\pi(S_ig_i)$, $i=1,\dots,s$, are pairwise disjoint;
\item[(iii)] the set $\Omega_0:=\mathcal{X}\setminus\bigsqcup_{i=1}^{s}\pi(S_ig_i)$ is
relatively compact.
\end{enumerate}
\end{prop}

\begin{proof}
This is essentially the content of the reduction theory of Garland and
Raghunathan \cite{GarlandRaghunathan}; since disjointness is not discussed there,
we provide the additional details needed to arrange it.
The finite set of elements of $G$ which parametrises the cusps in
\cite{GarlandRaghunathan} is not required to be rational; as $\Gamma$ is arithmetic,
the associated parabolic subgroups are rational, and
enlarging the Siegel set, this set may be
taken to consist of the representatives $g_1,\dots,g_s$ of \eqref{eq:cusps}. With this
choice, $\Gamma\cap g_i^{-1}Ug_i$ is a lattice in
$g_i^{-1}Ug_i$ for every $i$, and there exist $T_0\geq 0$, a relatively
compact open set $B\subset U$ and a relatively compact set $\Theta\subset G$ with
the following properties:
\begin{enumerate}
\item[(a)] $\big(\Theta\cup\bigcup_{i=1}^{s}\mathfrak{S}_{T_0,B}\,g_i\big)\Gamma=G$;
\item[(b)] if
$\mathfrak{S}_{T_0,B}\,g_i\gamma\cap\mathfrak{S}_{T_0,B}\,g_j\neq\emptyset$
for some $\gamma\in\Gamma$, then $i=j$ and $g_i\gamma g_i^{-1}\in\,{}^0P$.
\end{enumerate}
Enlarging $B$, we may assume that $B_i\subset B$ for all $i$.
The group $\Gamma_i$ acts on $KA_{T_0}U$ on the right as
$$
(k,t,x)\cdot mu_y=(km,\; t,\; m^{-1}x+y),\qquad mu_y\in\Gamma_i\subset MU .
$$
Hence $S_i=W_iA_{T_0}B_i$ meets every orbit in $KA_{T_0}U$ exactly once,
so that $\pi(S_ig_i)=\pi(KA_{T_0}Ug_i)$. Moreover $S_i\subset\mathfrak{S}_{T_0,B}$,
since $B_i\subset B$; hence, if two points of $S_ig_i$ have the same image under
$\pi$, they differ by an element $\gamma\in\Gamma$ with $g_i\gamma g_i^{-1}\in\Gamma_i$
by (b), and are therefore equal. This gives (i). Assertion (ii) is the first half of
(b), and (iii) follows from (a), since
$\pi(\mathfrak{S}_{T_0,B}g_i)\subset\pi(S_ig_i)$.
\end{proof}


The following lemma shows that, when
the parameters $\tau$ and $\sigma$ are small, only the vectors lying on the rational isotropic
lines can contribute to $\Omega(\tau,\sigma)$ deep in the cusps.

\begin{lem}\label{lem:small vectors}
There exists $\theta_0>0$ such that for all $\tau,\sigma\in(0,\theta_0)$, all
$i=1,\dots,s$, all $g\in KA_{T_0}Ug_i$ and all $z\in \mathcal{V}(\bZ)$,
\[
\text{if}\quad gz \in \mathcal{C}(\tau,\sigma)\enskip,\enskip \text{then}\quad
g_iz\in\bR_{>0} e_0\enskip .
\]
\end{lem}

\begin{proof}
Since $g_i$ is rational, there is $q_i\in\mathbb{N}$ with
$g_i\bZ^{d+2}\subset q_i^{-1}\bZ^{d+2}$, so that $L_+(g_iv)\in q_i^{-1}\bZ$ for every
$v\in\bZ^{d+2}$, and we set $\theta_0:=\big(2\sqrt{d+2}\big)^{-1}\min_{1\leq i\leq s}q_i^{-1}>0$.

Let $g=ka_tu_xg_i$ with $t>T_0\geq 0$, and let $z\in\mathcal{V}(\bZ)$ with
$gz\in\mathcal{C}(\tau,\sigma)$. Since $k$ is orthogonal, $\|a_tu_xg_iz\|=\|gz\|$,
and hence, using \eqref{eq:norm},
$$
L_+(a_tu_xg_iz)\leq 2\|a_tu_xg_iz\|_\infty\leq 2\|gz\|\leq
2\sqrt{d+2}\,\|gz\|_\infty< 2\sqrt{d+2}\,\max(\tau,\sigma).
$$
On the other hand, $L_+$ is invariant under $U$, so that by
\eqref{eq:scalling L_pm by a_t} for $t>0$,
$$
L_+(a_tu_xg_iz)=e^{t}L_+(g_iz)\geq L_+(g_iz).
$$
Since $\max(\tau,\sigma)<\theta_0$, we deduce that $L_+(g_iz)<q_i^{-1}$, whence
$L_+(g_iz)=0$, and therefore $g_iz\in\bR_{>0}e_0$, as observed after
\eqref{eq:scalling L_pm by a_t}.
\end{proof}

Using the parametrisation of a fundamental domain for $\Gamma$ from Proposition
\ref{prop:reduction}, we obtain the following exact formula for the
volume of the neighbourhoods of infinity in $\mathcal{X}$.

\begin{lem}\label{lem:volume Omega}
There exist $C_0>0$ and $\theta_1>0$, depending only on the dimension $d$, such that
\begin{equation}\label{eq:volume identity}
\mu\big(\Omega(\tau,\sigma)\big)=C_0\,\mu_{\mathcal{V}}\big(\mathcal{C}(\tau,\sigma)\big)
\qquad\hbox{for all $\tau,\sigma\in(0,\theta_1)$.}
\end{equation}
\end{lem}

\begin{proof}
Since $\Omega_0$ is relatively compact, Mahler's compactness criterion provides
$\rho>0$ such that $\|\lambda\|_\infty\geq\rho$ for every $\Lambda\in\Omega_0$ and
every $\lambda\in\Lambda$. As $\mathcal{C}(\tau,\sigma)\subset\{x\in\mathcal{V}:\,
\|x\|_\infty<\max(\tau,\sigma)\}$ by \eqref{eq:norm}, it follows that
$\Omega(\tau,\sigma)\cap\Omega_0=\emptyset$ whenever $\tau,\sigma\in(0,\rho)$. We set
$\theta_1:=\min\big(\theta_0,\,\rho,\,e^{-T_0}\min_ic_i\big)$ and assume from
now on that $\tau,\sigma\in(0,\theta_1)$.

By Proposition \ref{prop:reduction}, every
$\Lambda\in\Omega(\tau,\sigma)$ is of the form $\Lambda=hg_i\mathcal{V}(\bZ)$ for a
unique $i\in\{1,\dots,s\}$ and a unique $h=ka_tu_x\in S_i$, and the condition
$\Lambda\cap\mathcal{C}(\tau,\sigma)\neq\emptyset$ reads
$hg_iz\in\mathcal{C}(\tau,\sigma)$ for some $z\in\mathcal{V}(\bZ)$. By Lemma
\ref{lem:small vectors}, such a $z$ satisfies $g_iz\in\bR_{>0}e_0$, so
that $g_iz=mc_ie_0$ with $m\in\mathbb{N}$ by \eqref{eq:ci}. Since $u_xe_0=e_0$ and
$a_te_0=e^{-t}e_0$, we get $hg_iz=mc_ie^{-t}ke_0$; as $L_\pm$ are non-negative on
$\mathcal{V}$, the set $\mathcal{C}(\tau,\sigma)$ is stable under multiplication by
scalars in $(0,1]$, so that such an $m$ exists if and only if $m=1$ qualifies. We
conclude that
$$
\Omega(\tau,\sigma)=\bigsqcup_{i=1}^{s}
\pi\Big(\{ka_tu_xg_i\in S_ig_i:\; c_ie^{-t}ke_0\in\mathcal{C}(\tau,\sigma)\}\Big).
$$
Using Proposition \ref{prop:reduction}(i) and the right invariance of
$\mu_G$, and substituting $t=r+\ln c_i$, we obtain
\begin{align*}
\mu\big(\Omega(\tau,\sigma)\big)&=c_G\sum_{i=1}^{s}\int_{W_i}\int_{T_0}^{\infty}\int_{B_i}
\chi_{\{c_ie^{-t}ke_0\in\mathcal{C}(\tau,\sigma)\}}\; e^{-dt}\,dx\,dt\,d\mu_K(k)\\
&=c_G\sum_{i=1}^{s}\frac{\mathrm{vol}(B_i)}{c_i^{\,d}}
\int_{W_i}\int_{T_0-\ln c_i}^{\infty}
\chi_{\{e^{-r}ke_0\in\mathcal{C}(\tau,\sigma)\}}\; e^{-dr}\,dr\,d\mu_K(k),
\end{align*}
where $\mathrm{vol}$ denotes the Lebesgue measure on $U\cong\bR^d$.
If $e^{-r}ke_0\in\mathcal{C}(\tau,\sigma)$, then
$$
2e^{-r}=L_+(e^{-r}ke_0)+L_-(e^{-r}ke_0)<\tau+\sigma<2e^{-T_0}c_i,
$$
so that $r>T_0-\ln c_i$; the range of integration in $r$ may therefore be replaced by $\bR$.
Now \eqref{eq:volume identity} follows from
\eqref{decomposition muV}, with
$$
C_0:=c_G\sum_{i=1}^{s}\frac{\mathrm{vol}(B_i)}{|M_i|\,c_i^{\,d}}.
$$
\end{proof}

From now on we assume that the constant $\sigma_0$ of \eqref{eq:eta} satisfies
$\sigma_0<\theta_1$.

\section{Proof of the main results}
\label{sec:Poisson}
We show now how equidistribution of all orders can be used to establish Poisson asymptotics in the setting of shrinking targets in the space of orthogonal lattices. We first recall a very general criterion for Poisson approximation (Proposition \ref{prop:poisson criterion}) using the classical method of moments, then recall the exponential equidistribution of all orders of the measure $\nu$ towards the measure $\mu$ under the action of $A$, established by the second author in \cite{ouaggag2024clt}. The main result of this section is Theorem \ref{thm:poisson N}, which establishes that the number of visits to shrinking targets in an exponential equidistribution setting is asymptotically Poissonian; Theorems \ref{th:weibull} and \ref{th:poisson} follow from it by Lemmas \ref{lem:tilde kt tilde etaT} and \ref{lem:poisson tilde kt tilde etaT}. We follow an approach similar to that of \cite{WeibullBG}.

\subsection{General criterion for Poisson convergence}
We recall a general criterion for convergence to the Poisson distribution. Let $(\mathcal{Z},\nu)$ be a probability space (the criterion will be applied with
$\mathcal{Z}=\cX$ and the measure $\nu$ of Section \ref{sec:notation}) and let $H$ be an infinite set.  Let $(F_n)$ be a sequence of finite subsets of $H$ such that $|F_n| \ra \infty$.  Suppose that for every 
$n$ and $h \in F_n$,  we are given a measurable subset $A_{n,h} \subset \mathcal{Z}$.  Let
\begin{equation}
\label{def N general}
\mathcal{N}_{F_n}(z) := \#\big\{ h \in F_n \,  : \,  z \in A_{n,h} \big\},  \quad z \in \mathcal{Z}. 
\end{equation}
For positive integers $r,n$ and a real number $b \geq 0$, we define
\begin{equation}
\Delta_{n,r}(b) :=  \max\Big\{ \big| \,  |F_n|^r \cdot \nu\Big( \bigcap_{h \in F'} A_{n,h} \Big) - b^r \big| \,  : \,  F' \subset F_n, \enskip |F'| = r \Big\}.
\end{equation}
If $|F_n| < r$,   we set $\Delta_{n,r}(b) = 0$.  

\begin{prop}
[\cite{WeibullBG}]
\label{prop:poisson criterion}
Suppose that there exists $b \geq 0$ such that
\[
\lim_{n \ra \infty} \Delta_{n,r}(b) = 0,  \quad \textrm{for all $r \geq 1$}.
\]
Then $\mathcal{N}_{F_n} \xLongrightarrow[\nu]{} \Poi(b)$ as $n \ra \infty$.
\end{prop}

\subsection{Quantitative multiple equidistribution}\label{subsec:equidistribution}

We now recall the equidistribution estimate from \cite{ouaggag2024clt} in the form in
which it will be used; its error term is explicit in terms of the $C^l$-norms of
the test functions on $\mathcal{X}$, which we introduce first.

Every $Y \in \mathrm{Lie}(G)$ defines a first-order differential operator $D_Y$ on $C^{\infty}(\mathcal{X})$ by
$$
D_Y(\phi)(x) := \frac{d}{ds}\phi(\exp (sY)x)\big|_{s=0}.
$$

If $\{ Y_1,\dots, Y_k\}$ is a basis of $\mathrm{Lie}(G)$, then every monomial $Z=Y_1^{l_1}\dots Y_k^{l_k}$ defines a differential operator by 
\begin{equation}
\label{diff operator}
D_Z:= D_{Y_1}^{l_1}\dots D_{Y_k}^{l_k},
\end{equation}
of degree $\deg(Z)=l_1+\dots+l_k$. For integers $l\geq 0$ and $\phi \in C^{\infty}(\mathcal{X})$, we write
\begin{equation}\label{def C_l norm}
   \|\phi\|_{C^l}:= \sum_{\deg(Z)\leq l} \|D_Z(\phi)\|_{\infty}.
\end{equation}
The same notation is used for functions on $G$.

With this notation, the equidistribution estimate takes the following form.

\begin{thm}[\cite{ouaggag2024clt}]
\label{thm:equidistribution}
Let $d\geq 3$. For every $r\geq 1$, there exist $\beta_r>0$, $q_r\geq 1$ and $C_r>0$ such that for every $\varphi_1,\dots,\varphi_r \in C_{c}^{\infty}(\mathcal{X})$ and $t_1,\dots,t_r > 0$, 
$$
    \left|  \int_{\mathcal{X}} \left(\displaystyle\prod_{i=1}^r\varphi_i (a_{t_i} x)\right) d\nu(x) -  \prod_{i=1}^{r}\int_{\mathcal{X}} \varphi_i d\mu \right| \leq C_r e^{-\beta_r D_r(t_1,\dots,t_r)} \prod_{i=1}^{r}\|\varphi_i\|_{C^{q_r}},
$$
where 
\begin{equation} \label{eq:def spread D}
    D_r(t_1,\dots,t_r):= \min\{t_i,|t_i-t_j|:1\leq i\neq j\leq r\}.
\end{equation}
\end{thm}

We shall apply this estimate to functions in $C^{\infty}(\mathcal{X})$ with finite
$C^{q_r}$-norm, to which it extends by a routine approximation argument.

\subsection{Completion of the proofs}\label{subsec:completion}

Using Lemma \ref{lem:Weps}, we construct approximations for the sets $\mathcal{C}(\delta_n\tau,\sigma_0)$ for a sequence $\delta_n\to 0$.

\begin{lem}\label{lem:volume stability}
Let $\tau>0$, let $(\delta_n)$ be a sequence of positive numbers with $\delta_n\to 0$, and set $\eps_n:=\delta_n^2$. With $c>0$ as in Lemma \ref{lem:Weps} and
\begin{equation} \label{def Cn pm}
\mathcal{C}_n^{\pm} (\tau):= \mathcal{C}\big(\delta_n\tau\pm c\eps_n,\;\sigma_0\pm c\eps_n\big),
\end{equation}
the inclusions
\begin{equation}\label{eq:inclusion}
    W_{\eps_n}\cdot \mathcal{C}_n^-(\tau) \subset \mathcal{C}(\delta_n \tau,\sigma_0) \quad\quad\hbox{and}\quad \quad W_{\eps_n}\cdot  \mathcal{C}(\delta_n \tau,\sigma_0)\subset \mathcal{C}_n^+(\tau)
\end{equation}
hold for all $n$ large enough, and
\[
\lim_{n \ra \infty} \delta_n^{-d/2} \cdot \mu_{\mathcal{V}}\big(\mathcal{C}_n^{\pm}(\tau)\big) =C_d (\sigma_0\tau)^{d/2},
\]
with $C_d>0$ as in Lemma \ref{lem:volume regularity}.
\end{lem}

\begin{proof}
The inclusions \eqref{eq:inclusion} follow from Lemma \ref{lem:Weps} and the monotonicity of $\mathcal{C}$ in each of its parameters. By Lemma \ref{lem:volume regularity},
$$
\delta_n^{-d/2} \mu_{\mathcal{V}}\big(\mathcal{C}_n^{\pm}(\tau)\big)
=C_d\big((\tau\pm c\delta_n)(\sigma_0\pm c\delta_n^2)\big)^{d/2}
\longrightarrow C_d (\sigma_0\tau)^{d/2},
$$
as required.
\end{proof}

It will be convenient to extend the \emph{spread} $D_r$ defined in \eqref{eq:def spread D} to a function 
$\widetilde{D}_r$ defined on all finite subsets $\bT$ of $\bR_+$ by 
\begin{align*}
    &\widetilde{D}_r(\bT) := \min\big\{ D_r(t_1,\dots,t_r) \,  : \,  t_1,\dots,t_r \hbox{ distinct elements of } \bT\big\}\\
    \qand & \widetilde{D}_r(\bT) := 0, \text{ if } |\bT| < r.
\end{align*}
If a sequence $(\bT_n)$ satisfies the sparseness condition \eqref{eq:sparse},
then $\widetilde{D}_r(\bT_n)/\log |\bT_n|\ra\infty$ for all $r\geq 1$.
We shall use the counting function $\mathcal{N}_{\bT}$ of \eqref{def N} in the form
\begin{equation*}
\mathcal{N}_{\bT}(\Lambda ,\tau) =
\# \big\{ t \in \bT \,  : \,  a_t\Lambda \cap \mathcal{C}(\tau,\sigma_0) \neq \emptyset \big\},
\qquad \Lambda\in\cX,\;\tau> 0,
\end{equation*}
and can now state and prove the main result of this section.

\begin{thm}
\label{thm:poisson N}
Let $d\geq 3$. There exists $m>0$ such that, for every sequence $(\bT_n)$ of finite
subsets of $\bR_+$ satisfying the sparseness condition \eqref{eq:sparse},
we have
$$ 
\mathcal{N}_{\bT_n}(\cdot,|\bT_n|^{-2/d} \tau)\xLongrightarrow[\nu]{} \Poi(m\tau^{d/2}),
\quad\hbox{for all $\tau> 0$.}
$$
In particular,
$$
|\bT_n|^{2/d}\cdot\widetilde{\eta}_{\bT_n} \xLongrightarrow[\nu]{} \Wei(m^{-2/d},d/2).
$$
\end{thm}

\begin{proof}
    Fix $\tau>0$, set
$$
\delta_n:=|\bT_n|^{-2/d}\qand \eps_n:=\delta_n^2,
$$
and consider the hitting sets in $\mathcal{X}$ defined similarly to \eqref{def Omega} by
\begin{align*}
\Omega_n(\tau)&:= \{ \Lambda\in \mathcal{X}: \Lambda\cap \mathcal{C}_n(\tau)\neq \emptyset\},\\
\Omega_n^\pm(\tau)&:= \{ \Lambda\in \mathcal{X}: \Lambda\cap \mathcal{C}_n^{\pm}(\tau)\neq \emptyset\},
\end{align*}
where $\mathcal{C}_n(\tau)=\mathcal{C}(\delta_n \tau,\sigma_0)$ and $\mathcal{C}_n^\pm(\tau)$ are as in Lemma \ref{lem:volume stability}.

We take a family of non-negative smooth functions 
${\rho}_\eps$ on $G$
such that 
$$
\supp ({\rho}_\eps) \subset W_\eps,\quad\quad
\int_G {\rho}_\eps\,  d\mu_G = 1,\quad\quad
\|{\rho}_\eps\|_{C^q} \ll_q \eps^{-\gamma_q}
$$
for some $\gamma_q>0$. We set
$$
\varphi_n^- := \rho_{\eps_n} * \chi_{\Omega_n^-(\tau)}
\quad\hbox{and}\quad
\varphi_n^{+} := \rho_{\eps_n} * \chi_{\Omega_n^+(\tau)}.
$$
These are smooth functions on $\cX$ with finite $C^q$-norms, and the inclusions \eqref{eq:inclusion}
imply that
\[
\varphi_n^{-} \leq \chi_{\Omega_n(\tau)} \leq \varphi_n^{+},
\]
so that for  $\bT\subset \bT_n$
with $|\bT|=r$,
$$\int_\mathcal{X}\left( \prod_{t \in \bT} \varphi_n^{-} \circ a_t\right) \,  d\nu
\leq
\nu\left(\bigcap_{t \in \bT} a_t^{-1}\Omega_n(\tau)\right)
\leq \int_\mathcal{X} \left(\prod_{t \in \bT} \varphi_n^{+} \circ a_t\right) \,  d\nu.
$$
Furthermore,  since $\mu$ is $G$-invariant and $\int_G \rho_{\eps_n} \,  d\mu_G = 1$,  we have
$$\int_\mathcal{X} \varphi^{+}_n \,  d\mu = \mu(\Omega_n^{+}(\tau)) \qand \int_\mathcal{X} \varphi_n^{-} \,  d\mu = \mu(\Omega_n^{-}(\tau)). 
$$
With $q:=q_r$ as in Theorem \ref{thm:equidistribution} and $\eps_n=|\bT_n|^{-4/d}$,
$$\|\rho_{\eps_n}\|_{C^{q}(G)} \ll_q |\bT_n|^{\frac{4\gamma_{q}}{d}},
\quad\hbox{and hence}\;\;
\|\varphi^\pm_n\|_{C^{q}(\cX)} \ll_q |\bT_n|^{\frac{4\gamma_{q}}{d}}.
$$
Applying Theorem  \ref{thm:equidistribution}, we conclude that 
$$
\int_\mathcal{X} \left(\prod_{t\in \bT} \varphi_n^{\pm} \circ a_t\right) \,  d\nu= \mu(\Omega_n^\pm(\tau))^r+O_r\left(|\bT_n|^{\frac{4r\gamma_{q}}{d}} \cdot e^{-\beta_r \widetilde D_r(\bT_n)}\right),
$$
and
$$
|\bT_n|^r \nu\Big(\bigcap_{t \in \bT} a_t^{-1}\Omega_n(\tau)\Big)
\le \big(|\bT_n| \mu(\Omega_n^{+}(\tau))\big)^r  +O_r\left(|\bT_n|^{r+\frac{4r\gamma_{q}}{d}} \cdot e^{-\beta_r \widetilde D_r(\bT_n)}\right).
$$
This estimate is uniform over subsets $\bT\subset\bT_n$ with $|\bT|=r$.
One also obtains a similar lower bound in terms of
$\Omega_n^{-}(\tau)$.
We note that due to \eqref{eq:sparse}, the second term in these estimates converges to zero.
Furthermore, since $\delta_n\tau\pm c\eps_n\ra 0$ and
$\sigma_0\pm c\eps_n\ra\sigma_0<\theta_1$, Lemma \ref{lem:volume Omega} applies for all $n$
large enough and gives
$$
\mu(\Omega_n^{\pm}(\tau))= C_0\mu_{\mathcal{V}}(\mathcal{C}_n^{\pm}(\tau)),
$$
and by Lemma \ref{lem:volume stability},
$$
|\bT_n| \mu_{\mathcal{V}}(\mathcal{C}_n^\pm(\tau))
\to 
C_d\sigma_0^{d/2}\tau^{d/2}.
$$
Therefore, for every $r \geq 1$,
\[
\lim_{n \ra \infty} \max\Big\{ \Big| |\bT_n|^r \nu\Big(\bigcap_{t \in \bT} a_t^{-1}\Omega_n(\tau)\Big) - (m\tau^{d/2})^r \Big| \,  : \,  \bT \subset \bT_n,  \enskip |\bT| = r \Big\} = 0,
\]
where $m:=C_0C_d\sigma_0^{d/2}$.
We deduce from Proposition \ref{prop:poisson criterion}, applied with $\mathcal{Z}=\cX$,
$F_n=\bT_n$ and $A_{n,t}=a_t^{-1}\Omega_n(\tau)$, that
$$\mathcal{N}_{\bT_n}(\cdot,\delta_n\tau)\xLongrightarrow[\nu]{} \Poi(m\tau^{d/2}),
$$
which is the first assertion of the theorem. The second follows from it, since
$\{\widetilde{\eta}_{\bT_n}<\tau\}=\{\mathcal{N}_{\bT_n}(\cdot,\tau)>0\}$.
\end{proof}

\begin{proof}[Proof of Theorems \ref{th:weibull} and \ref{th:poisson}]
Set $m_d:=m\sigma_0^{-d/2}$, with $m$ as in Theorem \ref{thm:poisson N}. Applying Lemma
\ref{lem:poisson tilde kt tilde etaT} to the first assertion of
Theorem \ref{thm:poisson N}, we obtain
$$ \mathcal{K}_{\bT_n}(\cdot,|\bT_n|^{-1/d} \tau)\xLongrightarrow[\mu_{_{S^d}}]{} \Poi(m_d\tau^{d}),
$$
which is Theorem \ref{th:poisson}; and applying Lemma
\ref{lem:tilde kt tilde etaT} to the second assertion, with $\lambda=m^{-2/d}$, we obtain
Theorem \ref{th:weibull}, which also follows directly from Theorem \ref{th:poisson},
as explained in the introduction.
\end{proof}

\bibliographystyle{amsplain}
\bibliography{bibliography}

@article{alam2020quantitative,
  author  = {Alam, Mahbub and Ghosh, Anish},
  title   = {Quantitative rational approximation on spheres},
  journal = {Selecta Math. (N.S.)},
  volume  = {28},
  year    = {2022},
  number  = {5},
  pages   = {Paper No. 86, 25},
}

@article{WeibullBG,
  author  = {Bj{\"o}rklund, Michael and Gorodnik, Alexander},
  title   = {Poisson approximation and {W}eibull asymptotics in the geometry of numbers},
  journal = {Trans. Amer. Math. Soc.},
  volume  = {376},
  year    = {2023},
  number  = {3},
  pages   = {2155--2180},
}

@article{dani1985,
  author  = {Dani, S. G.},
  title   = {Divergent trajectories of flows on homogeneous spaces and {D}iophantine approximation},
  journal = {J. Reine Angew. Math.},
  volume  = {359},
  year    = {1985},
  pages   = {55--89},
}

@article{Doeblin1940,
  author  = {Doeblin, Wolfgang},
  title   = {Remarques sur la th{\'e}orie m{\'e}trique des fractions continues},
  journal = {Compositio Math.},
  volume  = {7},
  year    = {1940},
  pages   = {353--371},
}

@article{dolgopyat2021multipleborelcantellilemma,
  author  = {Dolgopyat, Dmitry and Fayad, Bassam and Liu, Sixu},
  title   = {Multiple {B}orel--{C}antelli lemma in dynamics and {M}ulti{L}og law for recurrence},
  journal = {J. Mod. Dyn.},
  volume  = {18},
  year    = {2022},
  pages   = {209--289},
}

@article{drutu2005,
  author  = {Dru{\c{t}}u, Cornelia},
  title   = {Diophantine approximation on rational quadrics},
  journal = {Math. Ann.},
  volume  = {333},
  year    = {2005},
  number  = {2},
  pages   = {405--469},
}

@article{GarlandRaghunathan,
  author  = {Garland, Howard and Raghunathan, M. S.},
  title   = {Fundamental domains for lattices in ({$\mathbb{R}$}-)rank 1 semisimple {L}ie groups},
  journal = {Ann. of Math. (2)},
  volume  = {92},
  year    = {1970},
  pages   = {279--326},
}

@article{Hirata_1993,
  author  = {Hirata, Masaki},
  title   = {Poisson law for {A}xiom {A} diffeomorphisms},
  journal = {Ergodic Theory Dynam. Systems},
  volume  = {13},
  year    = {1993},
  number  = {3},
  pages   = {533--556},
}

@article{KelmerYu2023,
  author  = {Kelmer, Dubi and Yu, Shucheng},
  title   = {Second moment of the light-cone {S}iegel transform and applications},
  journal = {Adv. Math.},
  volume  = {432},
  year    = {2023},
  pages   = {Paper No. 109270, 70},
}

@article{kleinbockmargulis1998,
  author  = {Kleinbock, D. Y. and Margulis, G. A.},
  title   = {Flows on homogeneous spaces and {D}iophantine approximation on manifolds},
  journal = {Ann. of Math. (2)},
  volume  = {148},
  year    = {1998},
  number  = {1},
  pages   = {339--360},
}

@article{Kleinbock_1999,
  author  = {Kleinbock, D. Y. and Margulis, G. A.},
  title   = {Logarithm laws for flows on homogeneous spaces},
  journal = {Invent. Math.},
  volume  = {138},
  year    = {1999},
  number  = {3},
  pages   = {451--494},
}

@article{kleinbock2013rational,
  author  = {Kleinbock, Dmitry and Merrill, Keith},
  title   = {Rational approximation on spheres},
  journal = {Israel J. Math.},
  volume  = {209},
  year    = {2015},
  number  = {1},
  pages   = {293--322},
}

@article{ouaggag2023effective,
  author  = {Ouaggag, Zouhair},
  title   = {Effective rational approximation on spheres},
  journal = {J. Number Theory},
  volume  = {249},
  year    = {2023},
  pages   = {183--208},
}

@misc{ouaggag2024clt,
  author = {Ouaggag, Zouhair},
  title  = {Effective estimate and {C}entral {L}imit {T}heorem for {D}iophantine approximation on spheres},
  year   = {2024},
  note   = {Preprint, arXiv:2409.02970},
}

@article{Pitskel1991,
  author  = {Pitskel, B.},
  title   = {Poisson limit law for {M}arkov chains},
  journal = {Ergodic Theory Dynam. Systems},
  volume  = {11},
  year    = {1991},
  number  = {3},
  pages   = {501--513},
}

@incollection{Pollicott2009,
  author    = {Pollicott, Mark},
  title     = {Limiting distributions for geodesics excursions on the modular surface},
  booktitle = {Spectral analysis in geometry and number theory},
  series    = {Contemp. Math.},
  volume    = {484},
  publisher = {Amer. Math. Soc.},
  address   = {Providence, RI},
  year      = {2009},
  pages     = {177--185},
}

@article{Sullivan1982,
  author  = {Sullivan, Dennis},
  title   = {Disjoint spheres, approximation by imaginary quadratic numbers, and the logarithm law for geodesics},
  journal = {Acta Math.},
  volume  = {149},
  year    = {1982},
  number  = {3-4},
  pages   = {215--237},
}

\end{document}